\documentclass{amsart}
\usepackage{amssymb,amscd}

 \newtheorem{thm}{Theorem}[section]
 \newtheorem{cor}[thm]{Corollary}
 \newtheorem{lemma}[thm]{Lemma}
 \newtheorem{prop}[thm]{Proposition}
 \theoremstyle{definition}
 \newtheorem{defn}[thm]{Definition}
 \newtheorem{exmp}[thm]{Example}
 \theoremstyle{remark}
 \newtheorem{rem}[thm]{Remark}

 \newtheorem{question}[thm]{Question}
 \numberwithin{equation}{subsection}
 \newtheorem{ack}{Acknowledgment}
 
\newcommand{\BB}{\text{$\mathcal{B}$}}

\newcommand{\UU}{\text{$\mathcal{U}$}}
\newcommand{\FF}{\text{$\mathcal{F}$}}
\newcommand{\GG}{\text{$\mathcal{G}$}}
\newcommand{\VV}{\text{$\mathcal{V}$}}
\newcommand{\WW}{\text{$\mathcal{W}$}}

\newcommand{\LL}{\text{$\mathcal{L}$}}

\newcommand{\LB}{\text{$\Lambda$}}
\newcommand{\sg}{\text{$\sigma$}}

\newcommand{\infim}{\operatorname{inf}}

\newcommand{\Cat}{\operatorname{Cat}}
\newcommand{\dm}{\operatorname{dim}}
\newcommand{\str}{\operatorname{sat}}
\newcommand{\suprem}{\operatorname{sup}}
\newcommand{\Fol}{\operatorname{Fol}}
\newcommand{\MeasFol}{\operatorname{MeasFol}}
\newcommand{\maxim}{\operatorname{max}}
\newcommand{\intr}{\operatorname{int}}
\newcommand{\Crit}{\operatorname{Crit}}
\newcommand{\id}{\operatorname{id}}

\newcommand{\Homeo}{\operatorname{Homeo}}

        \newcommand{\field}[1]{\text{$\mathbb{#1}$}}
        \newcommand{\N}{\field{N}}
        \newcommand{\Z}{\field{Z}}
        \newcommand{\Q}{\field{Q}}
        \newcommand{\R}{\field{R}}

\newdimen\theight
\def\TeXref#1{%
             \leavevmode\vadjust{\setbox0=\hbox{{\tt
                     \quad\quad  {\small \textrm #1}}}%
             \theight=\ht0
             \advance\theight by \lineskip
             \kern -\theight \vbox to
             \theight{\rightline{\rlap{\box0}}%
             \vss}%
             }}%

\begin{document}

\title{LS category  of laminations with transverse invariant measure}

\author{Carlos Meni\~no Cot\'on}

\address{Departamento de Xeometr\'{\i}a e Topolox\'{\i}a\\
         Facultade de Matem\'aticas\\
         Universidade de Santiago de Compostela\\
         15782 Santiago de Compostela}

\email{carlos.meninho@gmail.com}

\thanks{Supported by MICINN (Spain): FPU program and Grant MTM2008-02640}


\begin{abstract}
A version of the tangential LS category \cite{Macias} is introduced for topological laminations with a transverse invariant measure. Here, we use the transverse measure of the contraction of a tangential categorical open set instead of counting this set. This new measured category is invariant by leafwise homotopy equivalences preserving the transverse measures, and the condition of being zero or positive is a transverse invariant. The usual tangential LS category is also bounded by the number of certain critical sets. It is also proved that the measured category is semicontinuous when the foliated structure and the transverse invariant measure varies on a fixed manifold, which is a version of a result of W.~Singhof and E.~Vogt for the tangential category \cite{Vogt-Singhof}.

Hopefully, this relation between LS category and critical sets will be useful to deal with foliated versions of variational problems, like the existence of closed leafwise geodesics, and even to simplify the proofs of classical results on manifolds.
\end{abstract}

\maketitle

\section{Introduction}
The LS category is a homotopy invariant given by the minimum number of
open subsets contractible within a topological space needed to
cover it. It was introduced by L.~Lusternik and L.~Schnirelmann in the 1930's in the setting of variational problems. Many variations of this invariant has been given. In
particular, E.~Mac\'ias and H.~Colman introduced a tangential version for
foliations, where they use leafwise contractions to transversals
\cite{HellenColman, Macias}. In this paper, we define another
version of the tangential category where a transverse invariant
measure is used to ``count'' the number of those open sets by
measuring their leafwise deformations to transversals. It will be
computed in many examples, and important results of the classical
LS category are adapted to our setting.

The measured category is the infimum in the sums of the measures of deformations of open sets covering the lamination. It satisfies versions of some of the classical results of the LS category, like the cohomological lower bound by the length of the cup product, and the upper bounds given by the dimension or the number of critical points. For this last result, we work with Hilbert laminations (providing a general background for ``foliated'' variational problems); indeed, we get two results of that type: on the one hand, the tangential LS category of Mac\'ias-Colman is a lower bound for the number of critical sets of a differentiable function, where the critical sets are defined by using the leafwise gradient flow of a function, and, on the other hand, the measured LS category is a lower bound for the measure of the set of leafwise critical points (the critical points of the restrictions of the function to the leaves). The measured category is a lower semicontinuous map in the space of foliations with transverse invariant measure on a compact manifold, which is a version of a result for the tangential category \cite{Vogt-Singhof}. Finally, when the ambient space is compact, it is proved that the condition on the measured LS category to be zero or positive is a transverse invariant.

The measured category takes value zero easily (Corollary~\ref{c:minimalfoliation}). We compute it in the case of foliations by compact leaves providing, a setting where it can be useful. 

\section{Definition and first properties of the topological
\LB-category}

We refer to \cite{Candel-Conlon} for the basic notion and definitions about laminations, foliated chart, foliated atlas, holonomy pseudogroup and transverse invariant measure. They are recalled here in order to fix notations.

A {\em Polish space\/} is a second countable and complete metrizable space. Let $X$ be a Polish space. A {\em foliated chart\/} in $X$ is a pair $(U,\varphi)$ such that $U$ is an open subset of $X$ and $\varphi:U\to B^n\times S$ is an homeomorphism, where $B^n$ is an open ball of $\R^n$ and $S$ is a Polish space. The sets $B^n\times\{\ast\}$ are called the {\em plaques\/} of the chart, and the sets of the form $\varphi^{-1}(\{\ast\}\times S)$ are called the {\em associated transversals\/}. The map $U\to S$ is called the projection associated to $(U,\varphi)$. A {\em foliated atlas\/} is a family of foliated charts, $\{(U_i,\varphi_i)\}_{i\in I}$, that covers $X$ and the change of coordinates between the charts preserves the plaques; i.e., they are locally of the form $\varphi_i\circ\varphi_j^{-1}(x,s)=(f_{ij}(x,s),g_{ij}(s))$; these maps $g_{ij}$ form the holonomy cocycle associated to the foliated atlas. A {\em lamination\/} $\FF$ on $X$ is a maximal foliated atlas satisfying the above hypothesis. The plaques of the foliated charts of a maximal foliated atlas form a base of a finer topology of $X$, called the {\em leaf topology\/}. The connected components of the leaf topology are called the {\em leaves\/} of the foliation. The {\em dimension\/} of the lamination is the dimension of the plaques when all of them are open sets of the same Euclidean space.

A foliated atlas, $\UU=\{(U_i,\varphi_i)\}_{i\in \N}$, is called {\em regular\/} if it satisfies the following properties:
\begin{enumerate}
  \item [(a)] It is locally finite.
  \item [(b)] If a plaque $P$ of any foliated chart $(U_i,\varphi_i)$ meets another foliated chart $(U_j,\varphi_j)$, then $P\cap U_j$ is contained in only one plaque of $(U_j,\varphi_j)$.
  \item [(c)] If $U_i\cap U_j\neq \emptyset$, then there exists a foliated chart $(V,\psi)$ such that $U_i\cup U_j\subset V$, $\varphi_i=\psi|_{U_i}$ and $\varphi_j=\psi|_{U_j}$.
\end{enumerate}
Any topological lamination admits a regular foliated atlas, also we can assume that all the charts are locally compact. For a regular foliated atlas $\UU=\{(U_i,\varphi_i)\}_{i\in \N}$ with $\varphi_i:U_i\to B_{i,n}\times S_i$, the maps $g_{ij}$ generate a pseudogroup on $\bigsqcup_iS_i$. Holonomy pseudogroups defined by different foliated atlases are equivalent in the sense of \cite{Haefliger}, and the corresponding equivalence class is called the {\em holonomy pseudogroup\/} of the lamination; it contains all the information about its transverse dynamics. A {\em transversal\/} $T$ is a topological set of $X$ so that, for every foliated chart, $(U,\varphi)$, the corresponding projection restricts to a local homeomorphism $U\cap T\to S$. A transversal is said to be {\em complete\/} if it meets every leaf. On any complete transversal, there is a representative of the holonomy psudogroup which is equivalent to the representative defined by any foliated atlas via the projection maps defined by its charts. A {\em transverse invariant measure\/} is a measure on the space of any representative of the holonomy psudogroup invariant by its local transformations; in particular, it can be given as a measure on a complete transversal invariant by the corresponding representative of the holonomy pseudogroup. A {\em transverse set\/} is a measurable set that meets each leaf in a countable set. Any transverse invariant measure can be extended to all transverse sets so that they become invariant by measurable transformations that keep each point in the same leaf \cite{Connes}.

Observe that a measure is regular if the measure of a set $B$ is the infimum of the measures of the open sets containing it and the supremum of the measures of compact sets contained in $B$. Note that measures in the conditions of the Riesz representation theorem are regular, so it is not a very restrictive condition on measures.

\begin{rem}
Observe that the tangential model of the charts could be changed in order to give a more general notion of lamination: instead of taking open balls of $\R^n$ as $B^n$, we could take connected and locally contractible Polish spaces or separable Hilbert spaces. Also, it is possible to define the notion of $C^r$ foliated structure by assuming that the tangential part of changes of coordinates are $C^r$, with the leafwise derivatives of order $\le r$ depending continuously on the transverse coordinates. We can speak about regular atlases in Hilbert laminations but we cannot assume that its foliated charts are locally compact.
\end{rem}

Let us recall the definition of tangential category \cite{HellenColman,Macias}.
A lamination $(X,\FF)$ induces a foliated
measurable structure $\FF_U$ in each open set $U$. The space $U\times\R$ admits an
obvious foliated structure $\FF_{U\times\R}$ whose leaves are
products of leaves of $\FF_U$ and $\R$. Let $(Y,\GG)$ be another measurable lamination. A foliated map $H:\FF_{U\times\R}\to \GG$ is called a {\em tangential
homotopy\/}, and it is said that the maps $H(\cdot,0)$ and $H(\cdot,1)$ are {\em tangentially homotopic\/}. We use the term {\em
tangential deformation\/} when $\GG=\FF$ and $H( -
,0)$ is the inclusion map of $U$. A deformation such that $H( - ,1)$ is constant on the leaves of
$\FF_U$ is called a {\em tangential contraction\/} or an \FF-{\em
contraction\/}; in this case, $U$ is called a {\em tangentially categorical\/} or \FF-{\em categorical\/} open set. The {\em tangential category\/} is the lowest number of
categorical open sets that cover the measurable lamination. On one leaf foliations, this definition agrees with the classical category. The category of \FF\ is
denoted by $\Cat(\FF)$. It is clear that it is a tangential homotopy invariant.

Now, we introduce the relative category that will be useful for further applications.

\begin{defn}
Let $U\subset X$ be an open subset. The {\em relative category\/} of $U$, $\Cat(U,\FF)$, is the minimum number of \FF-categorical open sets in $X$ that cover $U$.
\end{defn}

\begin{rem}
Clearly, $\Cat(U,\FF)\leq\Cat(\FF_{U})$.
\end{rem}

\begin{prop}[Subadditivity of the relative category]\label{p:subaddtivity of the relative category}
Let $\{U_i\}_{i\in\N}$ be a countable family of open subsets of $X$. Then
\[
  \Cat\left(\bigcup_iU_i,\FF\right)\leq\sum_i\Cat(U_i,\FF)\;.
\]
\end{prop}

The following result is about the structure of a tangentially categorical open set.

\begin{lemma}[Singhof-Vogt~\cite{Vogt-Singhof}]\label{l:vogt0}
Let $\FF$ be a foliation of dimension $m$ and codimension $n$ on a
manifold $M$, let $U$ be an \FF-categorical open set, let $x\in
U$, let $D\subset U$ be a transverse manifold of dimension $n$, and
suppose that $x$ belongs to the interior of $D$. Then there exists
a neighborhood $E$ of $x$ in $D$ such that any leaf of $\FF_{U}$
meets $E$ in at most one point.
\end{lemma}

Observe that this lemma extends to the case of Hilbert laminations with the same proof. Hence, if the ambient space is separable, the final step of a tangential contraction is a countable union of local transversals of the foliation. Therefore it can be measured by a transverse invariant measure.

Let $\LB$ be a transverse invariant measure for
\FF\ and let $U$ be a tangentially categorical open set. Define
$$
  \tau_\LB(U)=\infim\{\,\LB(H(U\times\{1\})\mid H\ \text{is a tangential contraction of}\ U\,\}\;.
$$
Then the \LB-{\em category\/} of $(\FF,\LB)$ is defined as
$$\Cat(\FF,\LB)=\infim_{\UU}\sum_{U\in\UU}\tau_\LB(U)\;,$$
where \UU\ runs in the coverings of $X$ by \FF-categorical open sets. If the homotopies
used in this definition are required to be $C^r$ on leaves with the leafwise derivatives of order $\le r$ depending continuously on the transverse coordinates, then the term $C^r$ \LB-{\em
category\/} is used, with the notation $\Cat^r(\FF,\LB)$. Observe that, for one leaf laminations with the transverse invariant measure given by the counting measure, the measured category agrees with the classical LS category.

\begin{rem}
The map $\tau_\LB$ can be extended to any open set and tangential deformation. This is possible if we extend the transverse invariant measure to an ambient measure. For this purpose, we can use the {\em coherent extension\/} introduced in \cite{Menino1}. This measure is, roughly speaking, the integration of the counting measure in the leaves with respect to the transverse invariant measure. With this extension of $\tau$, the following holds: if $\LB$ is positive in non-empty open local transversals, then $\tau_\LB(U)<\infty\Rightarrow U\ \text{is \FF-categorical}$.
\end{rem}

The first property of the \LB- category that we prove is its homotopy invariance. A tangential homotopy equivalence $h$ from \FF\ to \GG\ induces a canonical
bijection between the sets of transverse invariant measures on \GG\ and \FF, which holds since $h$ induces an equivalence between the holonomy pseudogroups \cite{Alvarez-Masa}. For any tranverse invariant measurer $\LB$ of \GG, let $h^*\LB$ denote the transverse invariant measure in \FF\ induced by \LB\ and $h$. From now on, any tangential homotopy equivalence between laminations with transverse invariant measures is assumed to be compatible with the measures in the above sense.

\begin{lemma}[\cite{Menino3}]\label{l:invariance1}
Let $(X,\FF,\LB)$ and $(Y,\GG,\Delta)$ be a lamination with transverse invariant measures, and let $h:(X,\LB)\to (Y,\Delta)$ be a
measurable homotopy equivalence.
Then, for all transverse set $K\subset X$, $h(K)$ is also transverse and
$\Delta(h(K))\leq\LB(K)$.
\end{lemma}

\begin{prop}[The \LB-category is a tangential homotopy invariant]\label{p:invariance}
Let $(X,\FF,\LB)$ and $(Y,\GG,\Delta)$ be tangentially homotopy
equivalent laminations  with transverse invariant measures. Then $\Cat(\FF,\LB)=\Cat(\GG,\Delta)$.
\end{prop}

\begin{proof}
Let $h:(\FF,\LB)\to (\GG,\Delta)$ be a tangential homotopy equivalence, and let $g$ be a
homotopy inverse of $h$. Let $\{U_n\}$ ($n\in\N$) be a covering of $Y$ by
open sets. Then $\{h^{-1}(U_n)\}$ is a
covering of $X$ by open sets. We will prove that
$\tau_{\LB}(h^{-1}(U_n))\leq\tau_\Delta(U_n)$ for all $n\in\N$.
Let $H^n$ be a tangential contraction on each $U_n$,
and let $F$ be a tangential homotopy connecting the identity map and $g\circ h$.
Let
$$G=g\circ H\circ (f\times\id): h^{-1}(U)\times \R\to X\;.$$
Then $K:h^{-1}(U)\times \R\to X$, defined by
$$
K(x,t)=
\begin{cases}
F(x,2t) & \text{if $t\leq 1/2$} \\
G(x,2t-1) & \text{if $t\geq 1/2$\;,}
\end{cases}
$$
is a tangential contraction. Lemma~~\ref{l:invariance1} yields
$$\LB(K(h^{-1}(U)\times\{1\}))=\LB(g(H(U\times\{1\})))\leq\Delta(H(U\times\{1\}))\;.$$
Hence $\tau_{\LB}(h^{-1}(U_n))\leq\tau_\Delta(U_n)$ for all
$n\in\N$. Therefore $\Cat(\FF,\LB)\leq\Cat(\GG,\Delta)$. The
inverse inequality is analogous.
\end{proof}

The above proposition has an obvious $C^r$
version.

\begin{defn}
Let $(X,\FF,\LB)$ be a lamination with a transverse
invariant measure. A {\em null-transverse\/} set is a transverse set
$B$ such that $\LB(B)=0$.
\end{defn}

The following propositions are elementary.

\begin{prop}\label{p:nullset}
Let $(X,\FF,\LB)$ be a lamination with a transverse
invariant measure, and let $B$ be a null-transverse set. Then
$\Cat(\FF,\LB)$ can be computed by using only coverings of $X\setminus \str(B)$ by open sets in $X$, where $\str(B)$ denotes the saturation of $B$ in \FF. If $B$ is saturated and closed, then $\Cat(\FF,\LB)=\Cat(\FF_{X\setminus B},\LB_{X\setminus B})$.
\end{prop}

\begin{prop}
Let $T$ be a transversal such that $\# T\cap L\le \Cat(L)$ for all $L\in\FF$. Then $\LB(T)\leq \Cat(\FF,\LB)$. 
\end{prop}

\begin{proof}
Let $\{U_n\}_{n\in\N}$ be a tangential open covering of $\FF$ and let $H^n$ be tangential contractions for each $U_n$. Hence $T^n=H^n(U\times\{1\})$ are transverse sets and $\bigcup_n T^n$ meets all leaves. Therefore $\LB(T)\leq \Cat(\FF,\LB)$ since $U_n\cap L$ is categorical in $L$ for all $n\in\N$ which $U_n\cap L\neq\emptyset$ (see {\em e.g.\/} \cite{Vogt-Singhof}), and $\Lambda$ is holonomy invariant.
\end{proof}

Let $M$ be a manifold, $S$ a locally compact Polish space, $h:\pi_1(M)\to \Homeo(S)$ a homomorphism, and $\LB$ a measure on $S$ invariant by $h(\pi_1(M))$. Then the quotient space, $\widetilde{M}\times_h S$, of $\widetilde{M}\times S$, by the relation $(x,s)\sim (g^{-1} x, h(g)(s))$, $g\in\pi_1(M)$ is called the suspension of the homomorphism $h$, where $\widetilde{M}$ is the universal covering space of $M$. The suspension has a structure of lamination whose leaves are covering spaces of $M$, and $\LB$ is a transverse invariant measure.

\begin{prop}\label{p:boundsusp}
If $\FF$ is a lamination induced by a suspension, then $\Cat(\FF,\LB)\leq \Cat(M)\cdot \LB(S)$.
\end{prop}

\begin{prop}\label{p:products}
For a manifold $M$ and a locally compact Polish space $T$, let $M\times T$
be foliated as a product. Then $\Cat(M\times T,\LB)=\Cat(M)\cdot\LB(T)$ for every measure \LB\ on $T$, considered as a transverse invariant measure of $M\times T$ with the lamination with leaves $M\times\{*\}$.
\end{prop}

\begin{prop}\label{p:catsubspaces}
Let $\{U_n\}$ \upn{(}$n\in\N$\upn{)} be a covering by saturated open
sets of $(X,\FF,\LB)$. Then $\Cat(\FF,\LB)\leq\sum_n\Cat(\FF_{U_n},\LB_{U_n})$.
Here, the equality holds if $\{U_n\}$ is a partition.
\end{prop}

\begin{defn}
Let $(\FF,\LB)$ be a lamination with transverse invariant measure, the
{\em relative \LB-category\/} of $U$ is defined by
\[
\Cat(U,\FF,\LB)=\inf_{\UU}\sum_{V\in\UU}\tau_\LB(V)\;,
\]
where $\UU$ runs in the family of countable open coverings of $U$.
\end{defn}

\begin{rem}
Observe that $\tau_\LB(V)$ is defined by using
tangential homotopies deforming $V$ in the ambient space. Clearly,
$\Cat(U,\FF,\LB)\leq\Cat(\FF_{U},\LB)$.
\end{rem}

\begin{prop}[Subadditivity of relative
\LB-category]\label{p:subaddtivity} Let $\{U_i\}$ \upn{(}$i\in\N$\upn{)}
be a countable family of open subsets of $X$. Then
$\Cat\left(\bigcup_iU_i,\FF,\LB\right)\leq\sum_i\Cat(U_i,\FF,\LB)$.
\end{prop}

Now, we explain an important lemma that allows us to cut tangentially categorical open sets into small ones, in order to compute the \LB-category when the measure is finite on compact sets. In this process, a null transverse set is generated, but we know that this kind of set can be removed in the computation of \LB-category.

\begin{lemma}\label{l:measure finite in compact sets}
Let $(\FF,\LB)$ be a lamination with a transverse invariant measure that is finite on compact sets. Let $h:S\to T'$ be a holonomy map defined on a locally compact domain that can be extended to a neighborhood $W$ of $\overline{S}$. Then there exists an open neighborhood $T$ of $\overline{S}$ where $h$ is defined and such that $\LB(\partial T)=0$.
\end{lemma}

\begin{proof}
Let $d$ be a metric on $W$ that induces its topology. For each $\epsilon>0$, let
$$
  S^\varepsilon=\{\,x\in T\mid d(x,S)<\varepsilon\,\}\;.
$$ 
There exists $\delta>0$ such that $S^\delta\subset W$. The boundaries $\partial
S^\varepsilon$ are disjoint. If all of them have nontrivial
\LB-measure, then
\[
  \LB(S^\delta)\geq \suprem\left\{\,\sum_{\varepsilon\in I}\LB(\partial S^\varepsilon)\mid 
  I\subset (0,\delta),\ I\ \text{finite}\,\right\}
  =\sum_{0<\varepsilon<\delta}\LB(\partial U^\varepsilon)=\infty\;,
\]
where the last equality follows from the fact that any
non-countable sum of positive numbers is infinite. But
$\LB(\overline{S})$ is finite since \LB\ is finite on compact
sets.
\end{proof}

\begin{lemma}\label{l:disgregation}
Let $(\FF,\LB)$ be a lamination with a transverse invariant measure that is finite on compact sets. Let $U$ be a tangentially categorical open set and let $T$ be a complete transversal. Then there exist a partition $\{F,U_n\}_{n\in\N}$  of $U$ such that $F$ is a closed and null transverse set, each $U_n$ is open, and there exists a tangential homotopy $H:\bigcup_n U_n\times \R\to \FF$ such that $H(\bigcup_n U_n\times\{1\})\subset T$.
\end{lemma}

\begin{proof}
Let $G$ be a tangential contraction of $U$. By Lemma~\ref{l:vogt0}, $H(U\times\{1\})=S$ is a countable union of transversals. Therefore there exist a countable covering of $S$, $S=\bigcup_i S_i$, by transversals such that, for each $i\in\N$, there exists a holonomy map $h_i$ with domain $S_i$  and image contained in $T$ (since $T$ is a complete transversal). Each $h_i$ induce a tangential deformation $H^i:S_i\times \R\to\FF$ such that $H^i(S_i\times\{1\})=h_i(S_i)$. By Lemma~\ref{l:measure finite in compact sets}, we can suppose that $\LB(\partial S_i)=0$ for $i\in\N$.

Now, take $T_1=S_1$ and define recursively $T_n=\intr(V_n\setminus\bigcup_{i=1}^{n-1}S_i)$ and $K=U\setminus\bigcup_n T_n$. Of course, $K$ is closed in $S$ and it is null transverse since $K\subset\bigcup_i\partial S_i$. Finally, we obtain the partition by taking $U_n=H(-,1)^{-1}(T_n)$  and $F=H(-,0)^{-1}(K)$. The tangential contraction $H$ for $\bigcup_n U_n$ is defined as follows:
$$
H(x,t)=
\begin{cases}
G(x,2t) & \text{if $t\leq \frac{1}{2}$}\\
H^i(G(x,1),2t-1) & \text{if $t\geq\frac{1}{2}$ and $x\in U_i$\;.}\qed
\end{cases}
$$
\renewcommand{\qed}{}
\end{proof}

The way to define deformations used in the above prove will be useful in other sections of this work. This idea is specified by defining the following homotopy operation.

\begin{defn}\label{d:operation homotopy}
Let $H:X\times \R\to Y$ and $G:Z\times \R\to Y$ be tangential deformations (hence it is supposed that $X,Z\subset Y$, and $H(-,0)$ and $G(-,0)$ are inclusion maps) such that $H(X\times\{1\})\subset Z$. Then let $H*G$ be the the tangential deformation of $X$ defined by:
$$
H*G(x,t)=
\begin{cases}
H(x,2t) & \text{if $t\leq \frac{1}{2}$}\\
G(H(x,1),2t-1) & \text{if $t\geq\frac{1}{2}$\;.}
\end{cases}
$$
\end{defn}

\section{\LB-category of compact-Hausdorff laminations}

In this section, we compute the \LB-category of a lamination $(X,\FF)$ with all leaves compact and Hausdorff leaf space. Suppose that \LB\ is finite on compact sets.

In this setting, there is a nice description for the local
dynamics \cite{Epstein}. For a leaf $L$, there exists a
(topological) transversal such that there is an open foliated embedding of a suspension
$i_L:\widetilde{L}\times_{h_L} U_L\to (X,\FF)$, where $h_L:\pi_1(L)\times U_L\to U_L$ is an action defining the holonomy of $L$, with finite orbits and Hausdorff orbit space. A tranversal satisfying these conditions
for some leaf will be called a {\em slice\/}. We
omit $i_L$ in the notations for the sake of simplicity. The set of saturations of slices is a base of the saturated topology where the open sets of this topology is the saturation of open sets.

For compact laminations, it is also true that a transverse invariant measure induces a measure in the leaf space, which is defined in the following way. For each measurable subset $B\subset X/\FF$, there exists a transverse set $T\subset X$ such that $\pi(T)=B$ and $\pi_T:T\to B$ is a Borel isomorphism, where $\pi:X\to X/\FF$ is the canonical projection. Let $\LB_\FF(B)=\LB(T)$, which does not depend on the choice of such $T$ \cite{Edwards-Millett-Sullivan,Takesaki,Menino3}, and defines a measure $\LB_\FF$ on $X/\FF$.

\begin{rem}
If \LB\ is finite on compact sets, then, for any \sg-compact set
$K$ in a transversal, we have
$$\LB(K)=\infim\{\,\LB(V)\mid V\ \text{is open},\ K\subset V\,\}\;.$$
This means that \LB\ is {\em externally regular\/} on \sg-compact sets.
\end{rem}

\begin{prop}
For any leaf $L$, there exists a slice $U_L$ such
that $\LB(\partial U_L)=0$.
\end{prop}

\begin{proof}
Consequence of Lemma~\ref{l:measure finite in compact sets} and the fact that, according to notation of Lemma~\ref{l:measure finite in compact sets}, there exists an slice $U_L$ for each $L\in\FF$ and $\delta>0$ (depending on $L$) such that $U_L^\varepsilon$ is an slice for $0<\varepsilon<\delta$.
\end{proof}

\begin{thm}\label{t:topologicalcompact}
Let $(X,\FF,\LB)$ be a lamination with all leaves compact,
Hausdorff leaf space, and a transverse invariant measure that is finite on
compact sets. Then
$$\Cat(\FF,\LB)=\int_{X/\FF}\Cat(L)\,d\LB_\FF(L)\;.$$
\end{thm}

\begin{proof}
Let $\{B_0=X,\dots,B_\alpha,B_{\alpha+1},\dots\}$ be the Epstein filtration of
$(X,\FF)$ \cite{Epstein2}; it is defined by transfinite induction: $B_0\setminus B_1$ is the set of leaves with trivial holonomy, which is open and dense in $X$, and $B_\alpha$ is the set of leaves with non-trivial holonomy in $\bigcap_{\beta<\alpha}B_{\beta}$. This filtration is always countable, even though it may involve infinite ordinals $\alpha$. We obtain a partition of $X$ given by the sets $F_\alpha=B_{\alpha}\setminus B_{\alpha+1}$, which are laminations with all
leaves compact, trivial holonomy and dense in $B_\alpha$. Forgetting the correspondence between the inclusion relation on the family $\{B_\alpha\}$ and the order of the indices $\alpha$, we can change the indices of the family $\{F_\alpha\}$ to denote it by $\{F_i\}_{i\in\N}$. Given any $\varepsilon>0$, for each leaf $L\subset F_i$, choose a slice $U^i_L$ satisfying the following properties:

\begin{enumerate}
\item[(i)] $\overline{U^i_L}$ is compact;

\item[(ii)] $\LB(\partial U^i_L)=0$;

\item[(iii)] $\LB((\bigcup_{L\in F_i/\FF}U^i_L)\setminus
F_i)<\varepsilon/2^i$;

\item[(iv)]  $U^i_L$ meets each leaf of $F_i$ in at most one
point. The leaves of $F_i$ meeting $U^i_L$ are homeomorphic to
each other.
\end{enumerate}

Property (i) is given by the local compactness. Properties~(ii)
and (iii) are consequences of the external regularity on
\sg-compact sets, and property (iv) follows from the continuity of
every volume map on $F_i$ (see \cite{Epstein} for details). By the
Lindel\"off property, there exist leaves $L^i_1,\dots,L^i_n,\dots$ in
$F_i$ so that the family $\{L^i_n\times_{h^i_{L_n}}
U_{L^i_n}\}_{i,n\in\N}$ covers $X$. By induction on $n$, define
$A^i_1=L^i_1\times_{h_{L^i_1}}U_{L^i_1}$ and
$$
A^i_n=(L^i_n\times_{h_{L^i_n}}U_{L^i_n})\setminus(\overline{A^i_1}\cup\dots\cup
\overline{A^i_{n-1}})\;,\quad n>1\;.
$$
The families $\{A^i_n\}_{n\in\N}$ consist of disjoint open sets that cover
$F_i$ up to the saturation of a null-transverse set. Each $A^i_n$ is a suspension
with base $L^i_n$ and transverse fiber
$S^i_n=U_{L^i_n}\setminus(\overline{A^i_1}\cup\dots\cup\overline{A^i_{n-1}})$.
From these properties, it follows that
\begin{enumerate}

\item[(i)] $\LB(\bigcup_{n=1}^\infty S^i_n\cap
F_i)=\LB_\FF(F_i/\FF)$,

\item[(ii)] $\LB((\bigcup_{n=1}^\infty S^i_{n})\setminus
F_i)<\varepsilon/2^i$.

\end{enumerate}
Then, by Proposition~\ref{p:boundsusp} and using the dimensional
upper bound of the LS category of manifolds, we get
\begin{align*}
\Cat(\FF,\LB)&\leq\sum_{i=1}^\infty\sum_{n=1}^\infty\Cat(L^i_n)\cdot\LB(S^i_n)\\
&=\sum_{i=1}^\infty\int_{\bigcup_n S^i_n}\sum_{n=1}^\infty\Cat(L^i_n)\,\chi_{S^i_n}\,d\LB\\
&= \sum_{i=1}^\infty\int_{\bigcup_n S^i_{n}\cap
F_i}\sum_{n=1}^\infty\Cat(L^i_n)\,\chi_{S^i_n}\,d\LB \\
&\phantom{=\text{}}+ \sum_{i=1}^\infty\int_{\bigcup_n
S^i_{n}\setminus F_i}\sum_{n=1}^\infty\Cat(L^i_n)\,\chi_{S^i_n}\,d\LB\\
&\leq \sum_{i=1}^\infty\left(\int_{F_i/\FF}\Cat(L)\,d\LB_\FF(L) +
(\dim\FF+1)\cdot\LB\left(\left(\bigcup_{n=1}^\infty S^i_{n}\right)\setminus F_i\right)\right)\\
&<\int_{X/\FF}\Cat(L)\,d\LB_\FF(L) +
(\dim\FF+1)\cdot\varepsilon\;.\qed
\end{align*}
\renewcommand{\qed}{}
\end{proof}

\section{The positivity or nullity of the \LB-category is a transverse invariant}\label{s:invariance}
In this section, we show that the positivity or nullity of the \LB-category depends only on the holonomy pseudogroup and the transverse invariant measure; in fact, it can be considered as a property of pseudogroups with invariant measures, which is preserved by (measure preserving) equivalences of pseudogroups. We refer to \cite{Haefliger,Walczak} for the basic notions on pseudogroups. We recall here the definition of pseudogroup of homeomorphisms, equivalence of pseudogroups, and a definition of a good set of generators. Then we define the \LB-category of pseudogroups with invariant measures, showing that its nullity or positivity is preserved by (measure preserving) equivalences of pseudogroups, and it gives a lower bound for the \LB-category of foliations when the holonomy pseudogroup is considered.

First of all, we recall a result of the theory of topological dimension due to Milnor (see e.g. \cite{James}), which will be useful in the following sections of the paper.

\begin{prop}[Dimensional trick \cite{James}]\label{p:dimensional trick}
Let $X$ be a paracompact space of finite topological dimension and let $\UU$ be an open covering of $X$. Then there exists a covering $\{V_0,\dots,V_{\dim X}\}$ of $X$ such that each $V_i$ is a union of a countable family of disjoint open sets $V_{ij}$ and the open covering $\{V_{ij}\}$ is a refinement of $\UU$.
\end{prop}

\begin{defn}[Pseudogroup]
Let $T$ be a topological space. Let $\Gamma$ be a set of homeomorphisms between open subsets of $T$ such that $\id_T\in \Gamma$. It is said that $\Gamma$ is a {\em pseudogroup\/} if it is closed by the operations of composition (wherever defined), inversion, restriction to open sets and combination.
\end{defn}

Pseudogroups are generalizations of groups of transformations (or actions, or dynamical systems), and the most basic dynamical concepts have obvious versions for pseudogroups: orbits, saturated sets, etc. The restriction of a pseudogroup $\Gamma$ on $T$ to an open subset $U\subset T$ is the pseudogroup that consists of all elements of $\Gamma$ whose domain and image is contained in $T$. We suppose for the rest of the section that $T$ is a paracompact space of finite topological dimension and $\Gamma$ is finitely generated.

\begin{defn}[Equivalence of pseudogroups]
Let $\Gamma$ and $\Gamma'$ be pseudogroups of local homeomorphisms of topological spaces $T$ and $T'$, respectively. An ({\em \'etal\'e\/}) {\em morphism\/} $\Phi:\Gamma\to\Gamma'$ is a set of homeomorphisms from open sets of $T$ to open sets of $T'$ such that $\phi\circ g\circ \psi^{-1}\in\Gamma'$ for all $\phi,\psi\in\Phi$ and $g\in\Gamma$. $\Phi$ is called an {\em equivalence\/} if $\Phi^{-1}=\{\,\phi^{-1}\mid\phi\in\Phi\,\}$ is a morphism $\Gamma'\to\Gamma$; in this case, $\Gamma$ and $\Gamma'$ are said to be {\em equivalent\/}. The existence of an equivalence between $\Gamma$ and $\Gamma'$ is equivalent to the existence of a pseudogroup $\Gamma''$ on $T\sqcup T'$ such that $T$ and $T'$ meet all $\Gamma''$-orbits, and the restrictions of $\Gamma''$ to $T$ and $T'$ are $\Gamma$ and $\Gamma'$.
\end{defn}

\begin{defn}
Let $S$ be a symmetric set of generators of $\Gamma$. For each $n\in\N$, let $S_n$ denote the set of compositions of $n$ elements in $S$, and let $S_\infty=\bigcup_n S_n$. A {\em deformation\/} of $U$ is a map $h:U\to T$ that is combination of maps $h_i:U_i\to T$ in $S_\infty$ restricted to disjoint open sets; we may use the notation $h\equiv(h_i)$. It is said that $U$ is {\em deformable\/} if there is a deformation of $U$. The pairs $(U_i,h_i)$ are called {\em components\/} of $h$.
\end{defn}

Let $\Lambda$ be a $\Gamma$-invariant measure on $T$.

\begin{defn}
The \LB-category of a pseudogroup with an invariant measure, $(T,\Gamma,\LB)$, and a symmetric set of generators $S$ is defined as
$$
\Cat(\Gamma,\LB,S)=\inf_{\UU,h^U}\sum_{U\in\UU}\LB(h^{U}(U))\;,
$$
where $\UU$ runs in the collection of open coverings of $T$ by deformable sets, and $h^U\equiv(h^{U}_i)$ runs in the family of deformations of each $U\in\UU$.
\end{defn}

\begin{prop}\label{p:Cat(Gamma,LB,S)=0}
$\Cat(\Gamma,\LB,S)=0$ if and only if, for any $\varepsilon>0$, there exists an open set $U\subset T$ meeting any $\Gamma$-orbit and $\LB(U)<\varepsilon$.
\end{prop}

\begin{proof}
It is clear that $\Cat(\Gamma,\LB,S)=0$ means the existence of such open sets. To prove the reciprocal, let $U$ be an open set meeting any $\Gamma$-orbit and such that $\LB(U)<\varepsilon$. By Proposition~\ref{p:dimensional trick}, there exists a covering $\{U_0,\dots,U_{\dim T}\}$ of $T$ by deformable open sets in $U$ and deformations $h_i:U_i\to T$ such that $(h_i(U_i))\subset U$. Therefore $\Cat(\FF,\LB)\leq(\dim X + 1)\varepsilon$.
\end{proof}

\begin{cor}\label{c:independent of S}
The nullity or positivity of $\Cat(\Gamma,\LB,S)$ is independent of the choice of $S$.
\end{cor}

\begin{proof}
By Proposition~\ref{p:Cat(Gamma,LB,S)=0}, we have a characterization of $\Cat(\FF,\LB,S)$ by a condition on the measure of open sets that meet all orbits, which does not depend on the choice of generators.
\end{proof}

According to Corollary~\ref{c:independent of S},  instead of $\Cat(\mathcal{H},\LB,S)$, the notation $\Cat(\Gamma,\LB)$ will be for the conditions $\Cat(\Gamma,\LB)=0$ or $\Cat(\Gamma,\LB)>0$, which makes sense without any reference to the system of generators.

\begin{cor}\label{c:invariant by equivalences}
If $\LB$ is a regular measure, then the \LB-category of finitely generated pseudogroups on paracompact spaces of finite topological dimension is invariant by measure preserving equivalences.
\end{cor}

\begin{proof}
From the regularity of the measures, it easily follows that the existence of an open subset meeting all orbits with arbitrarily small measure is invariant by measure preserving psudogroup equivalences. So the result follows from Proposition~\ref{p:Cat(Gamma,LB,S)=0}.
\end{proof}

Consider a foliated manifold $(M,\FF)$ with a transverse invariant measure $\Lambda$, and let $\Gamma$ denote its holonomy pseudogroup acting on a complete transversal $T$ associated to a regular foliated atlas $\UU$ of $\FF$. A symmetric set of generators, $E^\UU$, is given by all holonomy maps induced by any non-empty intersection of a pair of charts in $\UU$.

\begin{prop}\label{p:inequalitycategories}
Let $\LB$ be a measure that is finite on compact sets and $\Gamma$-invariant on $T$. Then $\Cat(\Gamma,\LB,E^\UU)\le\Cat(\FF,\LB)$.
\end{prop}
\begin{proof}
Let $\{\,U_n\mid n\in\N\,\}$ be a covering of $M$ by tangentially categorical open sets and $H^n:U_n\times\R\to\FF$ be tangential deformations such that $\sum_n \LB(H(U_i\times\{1\}))<\Cat(\FF,\LB)+\varepsilon$. Without loss of generality, we can suppose that $H^n(U_n\times\{1\})\subset T$ by using Lemma~\ref{l:disgregation}. Hence $\{\,U_n\cap T\mid n\in\N\,\}$ is a covering of $T$ and $H^n(-,1):U_n\cap T\to T$ is a holonomy deformation.
\end{proof}

\begin{rem}
  In Proposition~\ref{p:inequalitycategories}, the reverse inequality does not hold in general, as shown by the following simple example. Consider a one compact leaf foliation; for instance, $S^1$. A complete transversal is given by one point $z\in S^1$, and the corresponding holonomy pseudogroup is $\Gamma=\{\id\}$. Choose the transverse invariant measure $\LB$ given by $\LB(\{z\})=1$. Then $\Cat(\Gamma,\LB)=1$, whilst $\Cat(S^1,\LB)=\Cat(S^1)=2$.
\end{rem}

\begin{prop}\label{p:equalitycategories}
Let $(M,\FF,\LB)$ be a foliated manifold of finite dimension with a transverse invariant measure, and let $\Gamma$ be its holonomy pseudogroup. If $\Cat(\Gamma,\LB)=0$, then $\Cat(\FF,\LB)=0$.
\end{prop}

\begin{proof}
Let $\UU=\{U_n\}\in\N$ be a regular foliated atlas, and let $T_j$ be a transversal associated to each foliated chart $U_j\in\UU$. For $\delta>0$, let $\{V^k\}_{k\in\N}$ be an open covering of $T$ and let $h^k\equiv(h^k_i)$ be a deformation of each $V^k$ such that $\sum_k\LB(\bigcup_i h^{k}_i(V^k_i))<\delta$. Let $\str_j(B)$ denote the saturation of $B$ in the chart $U_j\in\UU$. Let $\dim M=m$.

By Proposition~\ref{p:dimensional trick}, there exists a refinement $\BB$ of the covering $\{\str_j(T_j\cap V^k)\}_{j,k}$ such that any point in $M$ is contained in at most $m + 1$ sets in $\BB$, and there exists some point meeting $m + 1$ sets in $\BB$. We can take $\BB$ so that it can be subdivided into $m + 1$ families $\BB_0,\dots,\BB_{\dim M}$ of mutually disjoint open sets. Set $D_i=\bigcup_{B\in\BB_i} B$.

Each connected component of $D_i$ is contained in some of the open sets $\str_j(T_j\cap V^k)$, $1\leq j\leq K$, $k\in\N$. Now, it is easy to define a tangential contraction $H^i$ for each $D_i$ such that $\LB(H^i(D_i\times\{1\})<\delta$. It is enough to define $H^i$ on each connected component of $D_i$ since each connected component is contained in some $\str_j(T_j\cap V^k)$. Choose the minimum $k$ satisfying this condition for some $j$, and, then, choose the minimum $j$ satisfying this condition for that $k$. By connectivity, every connected component of $D_i$ is contained in only one of the open sets $\str_j(T_j\cap V^k_l)$ for $l\in\N$. Hence it can be tangentially contracted into the transversal $V^k_l$, and then it can be deformed to $h^k_l(V^k_l)$ by the tangential homotopy $G^k_l:V^k_l\times[0,1]\to M$ associated to the holonomy map $h^k_l$. With this tangential homotopy, it is clear that $\sum_i\LB(H^i(D_i\times\{1\})<\sum_{k\in\N}\LB(\bigcup_i h^{k}_i(V^k_i))<\delta$. Therefore $\Cat(\FF,\LB)< (m+1)\delta$.
\end{proof}

\begin{rem}
  By Corollary~\ref{c:invariant by equivalences}, the statement of Proposition~\ref{p:equalitycategories} is indeed valid for any representative of the holonomy pseudogroup.
\end{rem}

\begin{cor}\label{c:equalitycategories}
The nullity or positivity of the \LB-category of a foliated manifold of finite dimension with a regular transverse invariant measure $\Lambda$, which is finite on compact sets, is an invariant of the holonomy pseudogroup and the transverse invariant measure.
\end{cor}

\section{Dimensional upper bound}
A topological version of the dimensional upper bound for the tangential category is given in this section. The classical result due to W. Singhof and E. Vogt \cite{Vogt-Singhof} states that, for any $C^2$ lamination \FF\ on a compact manifold, $\Cat(\FF)\leq\dim\FF+1$.

\begin{prop}\label{p:homotopyholonomy}
Let $U$ be a tangentially categorical set, $T$ a complete
topological transversal, $H$ a tangential contraction for $U$ and
$\varepsilon>0$. Suppose that $\LB$ is finite in compact sets. Then there exists
two open sets $V,W$ covering $U$ and
tangential contractions $H^V,H^W$ of each $V$ and $W$ respectively such that
$$\LB(H^V(V\times\{1\}))+\LB(H^W(W\times\{1\}))\leq \LB(T) +
\varepsilon\;$$
\end{prop}

\begin{proof}
Observe that $T_H=H(U\times\{1\})$ is a transverse set (by Lemma~\ref{l:vogt0}); in
fact, for any $x\in U$, there exists an $\FF_U$-saturated open neighborhood $U_x\subset U$ such that $H(U_x\times\{1\})$ is a transversal associated to some foliated chart. Hence, since $T$ is a
complete transversal, by using Lemma~\ref{l:disgregation}, there exists a partition $\{F,V\}$ of $U$ such that $F$ is closed in $U$ and null transverse, and $V$ is an open set so that there exist a tangential contraction $H^V:V\times\R\to \FF$ satisfying $H(V\times\{1\})\subset T$. Now, observe that $\LB(H(F\times\{1\})=0$. Since $\LB$ is externally regular on \sg-compact sets (by the finiteness on compact sets), there exists an open neighborhood $S$ of $H(F\times\{1\})$ in $T_H$ such that $\LB(S)<\varepsilon$. Let $W=H(-,1)^{-1}(S)$. Clearly, these $V$ and $W$ satisfy the stated conditions.
\end{proof}

\begin{thm}
Let $(M,\FF,\LB)$ be a $C^2$ foliated compact manifold with a transverse
invariant measure that is finite in compact sets. Let $T$ be a complete transversal of $\FF$.
Then
$$\Cat(\FF,\LB)\leq (\dm\FF+1)\cdot\LB(T)\;.$$
\end{thm}

\begin{proof}
Since $\Cat(\FF)\leq\dm\FF+1$, we can use Proposition~\ref{p:homotopyholonomy} with
$\dm\FF +1$ tangentially categorical open sets covering $M$. Then
\[
  \Cat(\FF,\LB)\leq (\dm\FF+1)\cdot(\LB(T)+\varepsilon)
  \]
for all $\varepsilon>0$.
\end{proof}

\begin{cor}\label{c:minimalfoliation}
Let $(X,\FF)$ a minimal foliated manifold and let \LB\ be a
transverse invariant measure of \FF\ that is regular without atoms. Then $\Cat(\FF,\LB)=0$.
\end{cor}

\begin{rem}
The same arguments showed here can be used to prove the following more general result. Let $N\geq \Cat(\FF)$ and let $T$ be a complete transversal, then $\Cat(\FF,\LB)\leq N\cdot\LB(T)$, and this holds without further assumptions in the structure of the lamination (neither compactness nor any differentiable structure is needed).
\end{rem}

\begin{rem}
There exist minimal $C^1$-diffeomorphisms $f_\alpha$ of $S¹$ such that $rot(f_\alpha)= \alpha$ for $\alpha\in [0,2\pi]\setminus\Q$ and admitting a standard Borel fundamental domains $D_\alpha$ \cite{Kodama-Matsumoto}. The suspension of these diffeomorphisms yield minimal flows on the torus. A transverse invariant measure is determined by a usual measure in $D_\alpha$ extended to $S¹$ by $f_\alpha$. In measurable terms, this foliation is isomorphic to a product foliation $\R\times D_\alpha$. Therefore $\Cat(\R\times_{f_\alpha}S¹,\LB)>\LB(D_\alpha)>0$.   This is not a contradiction with Corollary~\ref{c:minimalfoliation} since the transverse invariant measures induced in $S¹$ are not regular. 
\end{rem}

\section{Upper semicontinuity of topological \LB-category}\label{s:upper semicontinuity}
We adapt a result of W. Singhof and E. Vogt \cite{Vogt-Singhof},
which asserts that the tangential category is an upper
semicontinuous map defined on the space of $C^2$ foliations over a
$C^\infty$ closed manifold $M$.

Consider an embedding of $M$ into an
Euclidean space $\mathbb{R}^{N}$. Then each $C^2$ foliation on $M$ can be identified to
a $C^1$-map $M\to \mathbb{R}^{N^{2}}$, which maps each point to
the orthogonal projection of $\mathbb{R}^{N}$ to the tangent space
of the foliation at this point. The topology on the space of
$C^{k+1}$ foliations over $M$ is induced by the topology of the
$C^{k}$-maps $M\to \mathbb{R}^{N^{2}}$. This topological space is
denoted by $\Fol_{p}^{k}(M)$.

Let $T\pitchfork \FF$ denote that $T$ is transverse to \FF\ (in
the differentiable sense).

\begin{prop}[Singhof-Vogt \cite{Vogt-Singhof}]\label{uso1} Let $M$ be a closed n-submanifold
of $\mathbb{R}^{N}$ and $\mathcal{F}\in \Fol_{p}^{1}(M)$. Let
$A_{1},\dots, A_{m}$ be compact $C^1$ $(n-p)$-submanifolds of $M$
transverse to $\mathcal{F}$, let $C_{i}\subset A_{i}$ be a collar
of the boundary of each $A_{i}$, and let
$A=\bigcup_{i}(A_{i}\setminus C_{i})$. Then there exists a
neighborhood $\VV$ of \FF\ in $\Fol_{p}^{1}(M)$ and an open
neighborhood $W$ of $A$ in $M$ such that $W$ is \GG-categorical
for all $\mathcal{G}\in \VV$.
\end{prop}

\begin{prop}[Singhof-Vogt \cite{Vogt-Singhof}]\label{uso2} Let $\mathcal{F}\in
\Fol_{p}^{1}(M)$, let $U$ be an open set in $M$, and let
$F~:~U~\times~I\to~M$ be a tangential homotopy such that
$F(x,0)=x$ for all $x\in U$. Let $\varepsilon>0$ and let $K$ be a
compact subset of $U$. Then there exists a neighborhood $\VV$ of
$\mathcal{F}$ such that, for all $\mathcal{G}\in \VV$, there
exists a \GG-tangential homotopy $G:U\times I\to M$ with $G(x,0)=x$ for all $x\in U$ and
$|F-G| <\varepsilon$ in $K\times I$.
\end{prop}

In the next proposition, $B_{r}^{k}$ denotes the open ball in $\R^k$ of radius $r$ and centered at the origin, and $D_{r}^{k}$ denotes the corresponding closed ball.

\begin{prop}[Singhof-Vogt \cite{Vogt-Singhof}]\label{atcon}
Let \FF\ be a $C^{r+1}$ foliation on a closed manifold
$M\subset\R^{N}$ with $\dm\FF=p$ and $\dm M=n$, and let $a\in M$.
Let $\varphi:B_{3}^{p}\times B_{4}^{n-p}\to U\subset M$ be a
parametrization associated to a foliated chart containing $a$.
Then there exists a neighborhood $\VV$ of \FF\ in
$\Fol_{p}^{r}(M)$ such that, for all $\GG\in \VV$ and $t\in
B_{3}^{n-p}$, there exists a map $g_{t}:B_{3}^{p}\to B_{4}^{n-p}$
with the following properties:
\begin{enumerate}

\item[(i)]  The map $g:B_{3}^{p}\times B_{3}^{n-p}\longrightarrow
B_{3}^{p}\times B_{4}^{n-p}$, defined by
$(x,t)\mapsto~(x,g_{t}(x))$, is a $C^{r+1}$-embedding.

\item[(ii)]  $B_{3}^{p}\times B_{2}^{n-p}$ is contained in the
image of $g$ and $B_{3}^{p}\times B_{1}^{n-p}$ is contained in
$g(B_{3}^{p}\times B_{2}^{n-p})$.

\item[(iii)]  For each $t\in B_{3}^{n-p}$, the set
$\{\,(x,g_{t}(x))\mid x\in B_{3}^{p}\,\}$ is contained in the leaf of
the pull back $\varphi^{\ast}\GG$ of \GG\, to $B_{3}^{p}\times
B_{4}^{n-p}$ through the point $(0,t)$.

\end{enumerate}
In fact, $g$ is uniquely determined by \upn{(}i\upn{)} and
\upn{(}iii\upn{)}, the map defined by $\varphi\circ
g:B_{3}^{p}\times B_{3}^{n-p}\to U_{g}\subset M$ is a $C^{r+1}$
parametrization of a foliated chart for \GG, and
$\varphi(B_{3}^{p}\times B_{2}^{n-p})\subset U_{g}$, where
$U_g=\varphi\circ g(B_{3}^{p}\times B_{3}^{n-p})$.
\end{prop}

\begin{defn}[Singhof-Vogt \cite{Vogt-Singhof}]\label{d:atcon}
Under the conditions of Proposition \ref{atcon}, the family
$\{\,\varphi\circ g\mid \GG\in \VV\,\}$ is called a {\em
simultaneous local parametrization\/} of $\VV$ in $M$.
\end{defn}

Simultaneous local parametrizations allow to work with foliations
near other ones in an easy way. We will use them to prove an
important lemma of this section.

\begin{exmp}\label{ex:Kronecker}
A previous motivation of the upper semicontinuity of the topological $\Lambda$-category is given by the Kr\"onecker flows $\mathcal{F}_\alpha$ on the torus, where $\alpha\in\R$ is the slope of the flow lines. The topological \LB-category of $\mathcal{F}_\alpha$ can be easily computed by using
Proposition~\ref{t:topologicalcompact} if $\alpha$ is rational, and it is
zero if $\alpha$ is irrational by Corollary~~\ref{c:minimalfoliation}.
We restrict our study to the subspace of these flows,
$\{\,\mathcal{F}_{\alpha}\mid \alpha\in\mathbb{R}\,\}\subset\Fol_{1}^{\infty}(T^{2})$, which is
homeomorphic to $\R$ by the mapping $\FF_{\alpha}\mapsto \alpha$. A transverse invariant measure $\Lambda$ for all of these flows is the normalized Lebesgue measure on a fixed
meridian. The topological
\LB-category is given by the mapping $\frac{m}{n}\mapsto
\frac{2}{|n|}$, $0\mapsto 2$ and $r\mapsto 0$, where $m$ and $n$ are
coprime integers and $r$ is irrational. Clearly, this map is upper
semicontinuous.

In the same example, we can also allow to vary the measure in a
continuous way by taking $(\FF_\alpha, f(\alpha)\cdot\LB)$, where
\LB\ is the normalized Lebesgue measure and $f~:~\R\to\R^+$ a
continuous map. The \LB-category function, in this case, is the
product of the above map and $f$, which is upper semicontinuous
as well.
\end{exmp}

Example~\ref{ex:Kronecker} suggests that the $\Lambda$-category is an upper
semicontinuous function in a certain topological space of foliations
with transverse invariant measures. From now on in this section, suppose that the transverse
invariant measures are finite on compact sets.

\begin{rem}\label{Vog2}
If $S$ is a complete differentiable transversal of a foliation
$\mathcal{F}$, then there exists a neighborhood
$U^{S}(\mathcal{F})$ of $\mathcal{F}$ in $\Fol_{p}^{1}(M)$ such
that $S$ is a complete differentiable transversal of all
$\mathcal{G}\in U^{S}(\mathcal{F})$.
\end{rem}

\begin{defn}[Strong topology]
Let $(\FF, \LB)$ be a $C^2$ foliation of dimension $p$ on a closed
manifold $M$ with a transverse invariant measure. Let $\UU$ be a
neighborhood of \FF\, in $\Fol^1_p(M)$, $T$ a differentiable
complete transversal of \FF, and $\varepsilon>0$. Let
$$
  U(\mathcal{F},\Lambda,\UU,T,\varepsilon)=\{\,(\GG,\Delta)\mid\mathcal{G}\in\UU,\ T\pitchfork\mathcal{G},\ \|\Lambda_{T}-\Delta_{T}\|<\varepsilon\,\}\;,
$$
where $(\GG,\Delta)$ is a $C^2$
foliation of dimension $p$ on $M$ with a transverse invariant
measure. This kind of sets form a base for a topology in the space
of $C^2$ foliations of dimension $p$ on $M$ with transverse
invariant measures finite on compact sets. This topological space
is denoted by $\MeasFol_{p}^{1}(M)$.
\end{defn}

For locally compact Polish spaces, there exists a weak topology on
the set of measures coarser than the norm topology. We use it to give
a weaker topology in the above space. It is said that a sequence $\{\LB_n\}$ of
measures on a locally compact Polish space $P$ converges in the
weak sense to a measure $\LB$ when the sequence $\int_P
f\,d\LB_n$ converges to $\int_P f\,d\LB$ for all continuous
function $f:P\to\R$.

\begin{defn}[Weak topology]
Let $(\FF_n,\LB_n)$ be a sequence of $C^2$ foliations with
transverse invariant measures on a closed manifold $M$. We say that
$(\FF_n,\LB_n)$ converges in the weak sense to $(\FF,\LB)$ if
$\FF_n$ converges to $\FF$ in $\Fol^1_p(M)$ and, for all complete
differentiable transversal $T$ to \FF\ and for all continuous
function $f:T\to \R$, the sequence $\int_T f\,d\LB_n$ converges to
$\int_T f\,d\LB$. The topological space determined by this
condition will be noted by $\WW\MeasFol_p^1(M)$.
\end{defn}

\begin{rem}\label{r:regularity}
Since measures are finite on compact sets, they are externally regular
 on \sg-compact sets. It is also true that they are {\em internally regular\/} on open sets; i.e., for all open set
$V$ contained in a transversal of \FF, we have
$$
  \LB(V)=\maxim\{\,\LB(K)\mid K\subset V\ \text{is a compact subset}\,\}\;.
$$
\end{rem}

We prove upper semicontinuity of the topological
\LB-category with respect to the weak topology, which implies its upper semicontinuity with respect to the strong topology too.
		


\begin{defn}\label{d:degreepreimage}
Let $f:A\to B$ be a map. For $n\in \mathbb{N}\cup\{\infty\}$, the $n$-{\em preimage\/} of $f$ is
$$
  F^{f}_{n}=\{\,x\in A\mid\# f^{-1}(f(x))=n\,\}\;.
$$
\end{defn}

In Definition~\ref{d:degreepreimage}, the sets $F^{f}_{n}$ are
\sg-compact when $f$ is a continuous map, and therefore they are
measurable.

\begin{defn}
A {\em local differentiable embedded transversal\/} is a transversal associated to a differentiable foliated chart; equivalently, it is a differentiable embedding $i:B^{n-p}(r)\to M$ such that $i(B^{n-p}(r))$ is contained in a distinguished open set and it is diffeomorphic to an open set of an associated transversal via the projection map.
An open set $U$ in a foliated space is called {\em regular\/}
when there exist a finite number of local differentiable
embedded transversals, $T_1,\dots,T_k$, contained in $U$ such that $T_1\cup\dots\cup T_k$ is a complete transversal of $\FF_U$. Suppose $\R^2$ with the usual foliation by lines $\R\times\{\ast\}$, consider the closed set $F=(\{0\}\cup\{\frac{1}{n}\mid n\in\N\})\times [0,\infty)$, the open set $\R^2\setminus F$ is not regular.
\end{defn}

\begin{rem}\label{r:openregular}
Let $U$ an open set in $M$ and $K\subset U$ a compact set. Then
there exists a regular open set $V$ such that $K\subset
V$ and $\overline{V}\subset U$. To prove this assertion, it is enough to
cover $K$ with a finite number of foliated charts whose closures
are contained in $U$.
\end{rem}

\begin{rem}\label{r:beginning}
Let $(M,\FF)$ be a $C^\infty$ foliated closed manifold, let $U$
be a regular \FF-categorical open set and let $H$ be an
\FF-contraction for $U$. By using Lemma~\ref{l:vogt0}, there exist
transversals $T_{1},\dots,T_{k}$ satisfying the following
conditions:
\begin{enumerate}

  \item[(1)] Any leaf of $\FF_{U}$ meets each $T_{i}$ at most in one
point and cuts some of them.

  \item[(2)] $H(T_{i}\times\{1\})$ is a local embedded transversal of \FF.

\end{enumerate}
By general position arguments, it is easy to prove that there
exists an \FF-contraction $G$ such that
$\LB(G(U\times\{1\}))\leq\LB(H(U\times\{1\}))$ and
$G(T_i\times\{1\})$ is a local differentiable embedded transversal
of \FF\ for $1\leq i \leq k$. Therefore we also suppose that
\begin{enumerate}

  \item[(3)] $H(T_{i}\times\{1\})$ is a differentiable transversal of \FF.

\end{enumerate}
Hence $H(U\times\{1\})=\bigcup_{i=1}^{k}H(T_{i} \times\{1\})$.
\end{rem}

The following lemma is obvious since $H(-,1):\bigcup_{i=1}^{k}T_{i}\to M$ is a continuous map.

\begin{lemma}\label{supra1}
Let $U$ be a regular \FF-categorical open set, let $H$ be an
\FF-contraction for $U$, let $T_{1},\dots,T_{k}$ be transverals
satisfying the previous conditions, and let \LB\; be a transverse
invariant measure. Then
\begin{align*}
\widetilde{\LB}(H(U\times\{1\}))=\sum_{i=1}^{k}\widetilde{\LB}(H(F^{H(
- , 1)}_{i}\times\{1\}))=\sum_{i=1}^{k}\frac{1}{i}\,\LB(F^{H( -
,1)}_{i})\;,
\end{align*}
where we consider $H( - ,1):\bigcup_{i=1}^{k}T_{i}\to M$, and
$F^{H( - ,1)}_{i}$ denotes de $i$-preimage of $H(-,1)$.
\end{lemma}

\begin{lemma}\label{l:ineqhomot}
Let $U$ be a regular \FF-categorical open set, let $H:U\times I\to
M$ be an \FF-contraction, let $T_{1},\dots,T_{k}$ be transversals
satisfying conditions of Remark~~\ref{r:beginning}, and let \LB\;
be a transverse invariant measure. Let $O$ be an open set
containing $H(U\times\{1\})=\bigcup_{i=1}^k H(T_i\times\{1\})$  such that
each leaf of $\FF_O$ meets $H(U\times\{1\})$. Let $G$ be an
\FF-contraction for $O$. Then
\begin{align*}
\widetilde{\LB}(H(U\times\{1\}))\geq\widetilde{\LB}(H*G(U\times\{1\}))\;.
\end{align*}
\end{lemma}

\begin{proof}
Observe that $\bigcup_{i=1}^j F^{H( -
,1)}_{i}\supset\bigcup_{i=1}^j F^{H*G( - ,1)}_{i}$ for $1\leq
j\leq k$, and then use Lemma~~\ref{supra1}.
\end{proof}

The following lemma has an analogous proof.

\begin{lemma}\label{c:monotcomp}
Under the same hypothesis, let $K$ be a \sg-compact subset of $O$.
Then
$$\widetilde{\LB}(H(U\times\{1\}))\geq\widetilde{\LB}(G(K\times\{1\}))\;.$$
\end{lemma}

\begin{prop}\label{3conditions}
Let $(\FF,\LB)\in\WW\MeasFol_{p}^{1}(M)$, let $U$ be a regular
\FF-categorical open set, let $H:U\times I\to M$ be an
\FF-contraction and let $W$ be an open set such that
$\overline{W}\subset U$. Let $T_{1},\dots,T_{k}$ be transversals
satisfying the conditions of Remark~\ref{r:beginning}, let $K$ be
a closed or open set in $\bigsqcup_{i=1}^{k}T_i$, and let
$\varepsilon,\delta>0$. Then, for every sequence $(\FF_n,\LB_n)$
converging to $(\FF,\LB)$ in $\WW\MeasFol_{p}^{1}(M)$, there exists
$N\in\N$ such that, $\forall n\geq N$,
\begin{enumerate}

\item[(i)] the transversals $T_1,\dots,T_k,H(T_{1}\times\{1\}),\dots,H(T_{k}\times\{1\})$ of \FF\ are also transversals of $\FF_n$;

\item[(ii)] $|\LB(K)-\LB_n(K)|<\varepsilon$; and

\item[(iii)] there exists a $C^1$ tangential homotopy $H^{\FF_{n}}$
for $\FF_n$ defined on $\overline{W}$ for all $n$ such that
$|H(x,t)-H^{\FF_n}(x,t)|<\delta$ for $x\in \overline{W}$.

\end{enumerate}
\end{prop}

\begin{proof}
Clearly,~(i) is a consequence of Remark~\ref{Vog2}. Observe that~(ii) is obtained by
using Urysohn's lemma and the regularity of
the measure (see Remark~\ref{r:regularity}). Finally,~(iii) is a direct consequence of
Proposition~\ref{uso2}.
\end{proof}

The following assertion is easy to prove. In the following, we use the notation $A_i=H(T_i\times\{1\})$.

\begin{lemma}\label{l:neighborhoods}
There exists a collar
$C_i$ of the boundary of $A_i$ such that
$H(\overline{W}\times\{1\})\subset\bigcup_{i=1}^k A_i\setminus
C_i$.
\end{lemma}

Let $P$ be an open set containing $\bigcup_{i=1}^k A_i\setminus
C_i$ under the conditions of Proposition~\ref{uso1}. By taking $N$
larger if necessary, $P$ is $\FF_n$-categorical for all $n\geq N$.
Let $G^{\FF_n}$ be an $\FF_n$-contraction for $P$.

\begin{prop}\label{p:open}
For $N$ large enough, there exists an open set
$P'\subset P$ such that $H(\overline{W}\times\{1\})\subset P'$ and
each leaf of $\FF_n|_{P'}$ meets $\bigcup_{i=1}^k A_i\setminus
C_i$ for all $n\geq N$.
\end{prop}

The proof of Proposition~\ref{p:open} is technical and will be given at
the end of this section.

\begin{lemma}\label{l:demi}
Let $U$ be a regular \FF-categorical open set, let $H$ be an
\FF-contraction of $U$ and let $W$ be an open set such that
$\overline{W}\subset U$. Then, for all $\varepsilon>0$ and for all
sequence $(\FF_n,\LB_n)$ converging to $(\FF,\LB)$ in
$\WW\MeasFol^{1}_p(M)$, there exists some $N\in\N$ and there exists
an $\FF_n$-contraction $J^{n}$ of $W$ for all $n\geq N$ such
that
\begin{align*}
\widetilde{\LB}(J^{n}(W\times\{1\}))\leq
\widetilde{\LB}(H(U\times\{1\}))+\varepsilon\;.
\end{align*}
\end{lemma}

\begin{proof}
By Remark~\ref{r:beginning}, Lemma~\ref{l:neighborhoods} and
Propositions~\ref{p:open} and~\ref{3conditions}, there exists
$N\in\N$ such that, for all $n \geq N$, there is a tangential
homotopy $H^{n}$ for $\FF^n$ with
$H^{n}(W\times\{1\})\subset P'$, where $P'$ is the open set
given by Proposition~\ref{p:open}. Moreover there exists an
$\FF_n$-contraction $G^{n}$ of $P'$ for $n\geq N$. We also can
suppose
$$
  \left|\LB_n\left(\bigcup_{i=1}^{k}A_i\right)-\LB\left(\bigcup_{i=1}^{k}A_i\right)\right|<\varepsilon
$$
for $n\geq N$ by Remark~\ref{r:beginning}-(3). Therefore, by
Lemma~\ref{c:monotcomp},
\begin{multline*}
  \LB(H(U\times\{1\}))+ \varepsilon=\LB\left(\bigcup_{i=1}^k H(T_i\times\{1\})\right)+\varepsilon\\
  \geq \LB_n\left(\bigcup_{i=1}^k H(T_i\times\{1\})\right)
  \geq\LB_n\left(\bigcup_{i=1}^k(A_i\setminus C_i)\right)\\
  \geq\LB_n\left(\bigcup_{i=1}^k P'\cap(A_i\setminus C_i)\right)
  \geq\LB_n\left(\bigcup_{i=1}^k H^{n}*G^{n}_1(W)\right)\;.\qed
\end{multline*}
\renewcommand{\qed}{}
\end{proof}

\begin{thm}[Upper semicontinuity of the topological \LB-category]\label{t:semi}The topological \LB-category map, $\Cat:\WW\MeasFol_{p}^{1}(M)\to\mathbb{R}$, is upper semicontinuous.
\end{thm}

\begin{proof}
We have to prove that, for all $\varepsilon>0$ and for any
sequence $(\FF_n,\LB_n)$ converging to $(\mathcal{F},\Lambda)$ in
$\WW\MeasFol_{p}^{1}(M)$, there exists $N\in\N$ such that
$\Cat(\FF_n,\LB_n)\leq \Cat(\mathcal{F},\LB) + \varepsilon$ for
all $n\geq N$. Take a finite covering $\{U_{1},\dots,U_{K}\}$ of $M$
by \FF-categorical open sets of $M$, and let $H^1,\dots,H^n$ be
\FF-contractions such that
$$
  \sum_{i=1}^K \LB(H^i(U_{i}\times\{1\})) \leq \Cat(\mathcal{F},\LB) + \frac{\varepsilon}{2}\;.
$$

By paracompactness and Remark~\ref{r:openregular}, we can suppose
that each $U_i$ is regular. By paracompactness again, there is an
open covering $\{W_{i}\}_{i=1,\dots,K}$ of $M$ such that
$\overline{W_{i}}\subset U_{i}$. By Lemma~\ref{l:demi}, there
exists $N\in\N$, and there are $\FF_n$-contractions
$H^{n,i}:W_i\times I\to M$ for $1\leq i\leq K$ and $n\geq N$
such that
 $$\LB_n(H^{n,i}(W_{i}\times\{1\}))\leq \LB(H^i(U_{i}\times\{1\}))+
\frac{\varepsilon}{2K}\;.$$ Finally, for all $n\geq N$, we get
\begin{multline*}
\Cat(\mathcal{F}_n,\LB_n)\leq
\sum_{i=1}^K\LB_n(H^{n,i}(W_i\times\{1\}))\\
\leq \sum_{i=1}^K\LB(H^i(U_i\times\{1\})) + \frac{\varepsilon}{2}
\leq \Cat(\mathcal{F},\LB) + \varepsilon\;.\qed\\
\end{multline*}
\renewcommand{\qed}{}
\end{proof}

Now, we prove Proposition~\ref{p:open} to conclude this section.
We restrict to the case where $T$ is a local differentiable
embedded transversal contained in a simultaneous local
parametrization. This case easily implies the general case.

Let us introduce a metric on $\Fol_{p}^{1}(M)$. Remember that a
$C^2$ foliation \FF\ can be identified to a $C^1$-map
$\FF:M\to\R^{N^2}\equiv\LL(\R^N,\R^N)$. For each $x\in M$, the corresponding linear map
$\FF_x:\R^N\to\R^N$ is the orthogonal projection of $\R^N$ onto
$T_x\FF$, where an embedding of $M$
into $\R^N$ is considered, and $T_x\FF\subset T_x\R^N\equiv\R^N$ denotes the subspace tangent to $\FF$ at $x$.

\begin{defn}
For foliations $\FF$ and $\GG$ on $M$ with the same dimension, let
\begin{align*}
d(\FF,\GG)=\maxim_{x\in M}\|\FF_x-\GG_x\|\;.
\end{align*}
This map $d$ is called the {\em foliated metric\/} on $\Fol_p^1(M)$.
\end{defn}

\begin{rem}
Clearly, $d$ is a metric on $\Fol_{p}^{1}(M)$ that induces its
topology; i.e., $\FF_n\to\FF$ in
$\Fol_{p}^{1}(M)$ if and only $d(\FF_n,\FF)\to0$.
\end{rem}

Since Proposition~\ref{p:open} is of local nature, we interpret
the foliated metric from the point of view of simultaneous local
parametrizations. Consider the notation of
Proposition~\ref{atcon}. Let $\Phi=\{\varphi\circ g\;|\;\GG\in
\VV\}$ be a simultaneous local parametrization around a point
$a=\varphi(0,0)\in M$, according to Definition~\ref{d:atcon}.

\begin{defn}
Let $\FF,\GG\in \VV$. Define the {\em $\varphi$-foliated metric\/} on $\varphi\in\Phi$
by
\begin{align*}
d_\varphi(\FF,\GG)=\maxim_{x\in D^p_3\times
D_1^{n-p}}\|\FF_x-\GG_x\|\;.
\end{align*}
\end{defn}

Obviously, $\FF_n$ converges to \FF\ in $\Fol^1_p(M)$ if and only if it converges to \FF\ with respect to the foliated metric on any
simultaneous local parametrization relative to \FF. Therefore we
can restrict to the simple case where $M=\R^n$ and $\FF$  is the
standard foliation by parallel planes with dimension $p$. For this,
we take a foliated diffeomorphism $\alpha:\intr(D^p_3\times D_1^{n-p})\to\R^n$ with $\alpha(0)=0$.

\begin{defn}
Let \FF\ be the foliation of dimension $p$ on $\R^n$ defined by
planes parallel to the plane $x_{p+1}=\dots=x_n=0$. For all
$\varepsilon,\delta>0$, let
\begin{multline*}
  R_0(\varepsilon,\delta)\\
  =\{\,(x_1,\dots,x_n)\mid\|(x_{p+1},\dots,x_n)\|\leq\varepsilon\|(x_1,\dots,x_p)\|,\ \|(x_1,\dots,x_n)\|\leq\delta\,\}\;.
\end{multline*}
For $a\in M$, let $R_a(\varepsilon,\delta)=a+R_0(\varepsilon,\delta)$.
\end{defn}

\begin{rem}
We have $d(\FF,\GG)<\varepsilon$ if and only if $T_a\GG\subset
R_a(\varepsilon,\infty)\ \forall a\in \R^n$.
\end{rem}

The following Lemma is easy to prove.

\begin{lemma}\label{l:analysis}
Let $f:\R^p\to\R^{n-p}$ be a $C^1$ map such that $f(0)=0$ and the
graph of its differential map, $df(x)$, is contained in
$R_0(\varepsilon,\infty)$ for all $x\in\R^p$. Then the graph of
$f$ is contained in $R_0(\varepsilon,\infty)$.
\end{lemma}

\begin{rem}\label{r:result}
In our setting, the graph of $f$ represents a leaf near \FF.
Therefore, for $(v,w)\in B^p_3\times B_1^{n-p}$, there exists a
compact region of the form $\overline{\alpha^{-1}(R_{\alpha(v,w)}(\varepsilon,\infty))}$ such that any leaf in $D^p_3\times D_1^{n-p}$ of $\varphi^*\GG$ through $(v,w)$
is contained in this region (by taking \VV\ smaller if necessary).
\end{rem}
Now, we can prove Proposition~\ref{p:open}. Consider the
particular case $T=\{0\}\times B_{1/2}^{n-p}$, which is an
embedded smooth local transversal contained in $B^p_3\times
B_1^{n-p}$. Define the {\em $\rho$-tube\/} of $T$ by $U(T,\rho)=
B^p_{\rho}\times B_{1/2}^{n-p}$, $0<\rho< 3$. The boundary
$\partial T$ is compact, and, for all $x\in\partial T$, let $R_x$
be a region satisfying the conditions of Remark~\ref{r:result}.
The set $K=\bigcup_{x\in\partial T}R_x$ is compact and therefore
closed. The set $P(T,\rho)=U(T,\rho)\setminus K$ is an open
neighborhood of $T$ and satisfies the required conditions for a
ball centered at \FF\ with small radius with respect to the
foliated metric.

\begin{proof}[Proof of Proposition \ref{p:open}]
Under the conditions of Proposition~\ref{p:open}, there exist $k$
transversals, $S_1,\dots,S_k$, such that $\overline{S_i}\subset
A_i\setminus C_i$ for $1\leq i\leq k$, and
$H(\overline{W}\times\{1\})\subset \bigcup_{i=1}^k S_i$. Observe
that any smooth embedded local transversal is associated to a
foliated chart. Therefore, if $P$ is a neighborhood of
$\bigcup_{i=1}^k A_i\setminus C_i$, there exists a neighborhood
$\VV$ of \FF\ in $\Fol^1_p(M)$ and there exists a neighborhood
$P(S_i,\rho_i)\subset P$ of each $S_i$ such that each leaf of
$\GG_{P(S_i,\rho_i)}$ meets $S_i$ (and therefore meets $A_i\setminus
C_i$) for all $\GG\in\VV$. The open set $\bigcup_{i=1}^k
P(S_i,\rho_i)$ satisfies the required conditions.
\end{proof}

The upper semicontinuity can be a tool to get lower bounds for the
\LB-category of a $C^2$ foliation on a closed $C^\infty$ manifold: if $(\FF_n,\LB_n)\to(\FF,\LB)$ with respect to the
weak topology, then $\limsup\Cat(\FF_n,\LB_n)\leq\Cat(\FF,\LB)$.

\begin{rem}
Measures satisfying the conditions of the Riesz Representation Theorem are regular. Therefore regularity is not a very restrictive condition on the measures.
\end{rem}


\section{Critical points}

In this section, we adapt a theorem due to  J. Schwartz \cite{Schwartz}, which states that the LS category of a separable Hilbert manifold is a lower bound of the number of critical points of any Palais-Smale function. In this section, we work with Hilbert  laminations. Thus, the foliated charts are homeomorphisms to $H\times P$, where $H$ is a separable Hilbert space and $P$ is a Polish space. The ambient space of the foliation is a Polish space, and therefore we can work with countable foliated atlases. We hope that the work of this section will be useful to study laminated versions of variational problems where the classical Lusternik-Schnirelmann category was applied \cite{Ballman,Ballman-Thorbergsson-Ziller,Klingenberg,Lusternik-Schnirelmann}. Our result is a direct corollary of a similar adaptation for usual tangential category \cite{Menino2}.

Moreover we consider $C^2$ Hilbert laminations; i.e., the change of foliated coordinates are leafwise $C^2$ whose tangential derivatives of order $\le2$ are continuous on the ambient space. Also, any lamination is assumed to have a locally finite atlas such that each plaque of every chart meets at most one plaque of any other chart. We consider a leafwise Riemannian metric  so that its tangential derivatives of order $\le2$ are continuous on the ambient space; thus each leaf becomes a Riemannian Hilbert manifold. Of course, the holonomy pseudogroup makes sense in this set-up, as well as transverse invariant measures. We only consider regular measures. Here, an open transversal is an embedded space that is locally homeomorphic to a transversal associated to a foliated chart via a projection map.

We consider functions that are $C^2$ on the leaves whose tangential derivatives of order $\le2$ are continuous on the ambient space. The functions satisfying the above property are called $C^r$, and they form a linear space denoted by $C^{2}(\FF)$. For a function $f\in C^{2}(\FF)$, we set $\Crit_{\FF}(f)=\bigcup_{L\in\FF}\Crit(f|_L)$. The definition of {\em $C^r$ Hilbert laminations\/} and their $C^r$ functions.

\begin{exmp}[Construction of a tangential isotopy \cite{Palais}]\label{e:vectorfield}
A tangential isotopy can be constructed on a Hilbert manifold by using a $C^1$ tangent vector field $V$. There exists a flow $\phi_t(p)$ such that $\phi_0(p)=p$, $\phi_{t+s}(p)=\phi_t(\phi_s(p))$ and $d\phi_t(p)/dt=V(\phi_t(p))$. From the way of obtaining $\phi$ \cite{Palais,Crandall-Pazy}, it follows that the same kind of construction for a $C^1$ tangent vector field on a measurable Hilbert lamination $(X,\FF)$ induces a tangential isotopy on $(X,\FF)$.

Now we obtain a tangential isotopy from the gradient flow of a differentiable map. It will be modified by a control function $\alpha$ in order to have some control on the deformations induced by the corresponding isotopy.
Let $\nabla f$ be the gradient tangent vector field of $f$; i.e., the
unique tangent vector field satisfying $df(v)=\langle v,\nabla
f\rangle$ for all $v\in T\FF$. Take the $C^{1}$ vector field $V=-\alpha(|\nabla f|)\,\nabla f$, where $\alpha:[0,\infty)\to\R^+$ is $C^\infty$, $\alpha(t)\equiv 1$ for $0\leq t\leq 1$, $t^2\alpha(t)$ is monotone non-decreasing and $t^2\alpha(t)=2$ for $t\geq 2$. The flow $\phi_t(p)$ of $V$ is defined for $-\infty<t<\infty$ \cite{Schwartz}, and it is called the {\em modified gradient flow\/}.
\end{exmp}

Let us define a partial order relation ``$\ll$'' for the critical
points of $f$. First, we say that $x<y$ if there exists a regular
point $p$ such that $x\in\alpha(p)$ and $y\in\omega(p)$, where
$\alpha(p)$ and $\omega(p)$ are the $\alpha$- and $\omega$-limits of
$p$. Thendefine $x\ll y$ if there exists a finite sequence of critical
points, $x_1,\dots,x_n$, such that $x < x_1 < \dots< x_n < y$.

\begin{lemma}\label{l:toplowaproxim}
Let $T\subset X$ be a transversal meeting each leaf in a discrete set, and let $\varepsilon>0$. Then there exists a \FF-categorical open set $U$ containing $T$ such that
$\Cat(U,\FF,\LB)< \LB(T) + \varepsilon$.
\end{lemma}

\begin{proof}
It is easy to see that we can suppose that $\overline{T}$ is contained in a foliated chart $V$. Let $T_V$ be a transversal associated to $V$ and $\pi:V\to T_V$ the canonical projection. Since $\LB$ is regular, there exists an open subset $F\subset T_V$ containing $\pi(T)$ such that $\LB(F)<\LB(\pi(T))+ \varepsilon$. Of course, $T\subset\str_V(F)$ and $\str_V(F)$ is an open set that contracts tangentially to $F$. The result follows since $\LB(\pi(T))\leq \LB(T)$.
\end{proof}

We restrict our study to the case where the leafwise critical points are isolated on the leaves or, in the case of the \LB-category, the set of leaves with non leafwise isolated critical points is null transverse. Therefore the set of leafwise critical points is a transverse set.

\begin{defn}[Critical sets]\label{d:partcrit}
Suppose that $f:\FF\to\R$ is a $C^2$ function such that for each leafwise critical point $p$, either $p$ is a relative minimum, or there exists another critical point $x$ such that $p<x$. Define $M_0,M_1,\dots$ inductively by
\begin{align*}
M_0&=\{\,x\in \Crit_\FF(f)\mid\not\exists\ y\ \text{such that}\ x\ll y\,\}\;,\\
M_i&=\{\,x\in\Crit_\FF(f)\mid\forall y\,\ x\ll y\ \Rightarrow\ y\in M_0\cup\dots\cup M_{i-1}\,\}\;.
\end{align*}
Clearly, $M_0$ contains all relative minima on the leaves. We also
set $C_0(f)=M_0$ and $C_i(f)=M_i\setminus(M_0\cup\dots\cup
M_{i-1})$. Observe that, if $x\in\omega(p)\cap C_i(f)$ for
some $i$, then $\omega(p)\subset C_i(f)$. There is an analogous
property for the $\alpha$-limit. The notation $C_i$ will be used
if there is no confusion and they will be called the {\em critical sets\/} of $f$. Let $p\ll p^*$. Then $i_{p^*}< i_p$,
where $i_{p}$ and $i_{p^*}$ are the indexes such that $p\in
C_{i_p}$ and $p^*\in C_{i_{p^*}}$. Observe that the critical sets are $\sg$-compact.
\end{defn}

The definition of a Palais-Smale condition is needed for this section. For Hilbert laminations, it could be adapted by taking functions that satisfy the Palais-Smale condition on all (or almost all) leaves. But this is very restrictive because it would mean that the set of relative minima meets each leaf in a relatively compact set (which is non-empty when $f$ is bounded from below), and therefore there would exist a complete transversal meeting each leaf at one point. This would be a very restrictive condition on the foliation. Thus, instead, we use the following weaker condition.

\begin{defn}\label{d:PStopological}
An {\em $\omega$-Palais-Smale\/} (or simply, {\em $\omega$-PS\/}) function is a function $f\in
C^{2}(\FF)$ such that all of its critical sets are closed (in the ambient topology), for any $p\in\Crit_\FF(f)$, the set $\{\,x\in\Crit_\FF(f)\mid p\ll x\,\}$ is compact, and this set is empty if and only if $p$ is a relative minimum and any flow line of $\phi$ has a non-empty $\omega$-limit. An {\em $\alpha$-Palais-Smale\/} (or simply, {\em $\alpha$-PS\/}) function is defined analogously by taking the set $\{\,x\in\Crit_\FF(f)\mid x\ll p\,\}$.
\end{defn}

Of course, $f$ is $\omega$-PS if and only if $-f$ is $\alpha$-PS. The set of relative minima of a $\omega$-PS function bounded from below is always non-empty in any leaf. A usual bounded from below Palais-Smale function on a manifold with isolated critical points is clearly an $\omega$-PS function.

\begin{thm}\label{t:TanCrit}\cite{Menino2}
Let $(X,\FF)$ be a Hilbert lamination endowed with a Riemannian metric on the leaves varying continuously on the ambient space, and let $f$ be an $\omega$-PS function. Suppose that $\Crit_\FF(f)$ meets each leaf in a discrete set. Then
$\Cat(\FF)\leq \#\{\text{critical sets of}\ f\}\;.$ In fact there exist a covering by tangentially categorical open sets $U_i$ such that each $U_i$ contracts to a transverse object that is a transverse neighborhood of the critical set $C_i$ (so small as desired).
\end{thm}

\begin{cor}\label{t:meascrit}
Let $(X,\FF,\LB)$ be a Hilbert lamination endowed with a
Riemannian metric on the leaves varying continuously on the ambient space, and with a (regular) transverse
invariant measure. Let $f$ be an $\omega$-PS function. Then $\Cat(\FF,\LB)\leq \LB(\Crit_\FF(f))$.
\end{cor}

\begin{proof}[Proof of Corollary~\ref{t:meascrit}]
We suppose that leafwise critical points are leafwise isolated. Take $U_i$ like in Theorem~\ref{t:TanCrit} so that $\Cat(U_i,\FF,\LB)\leq \LB(C_i) + \varepsilon/2^i$ (by Lemma~\ref{l:toplowaproxim}).
\end{proof}

\begin{question}
We can ask if the same results are also true when the critical sets are not closed. We are greatly convinced that the answer is affirmative, but the proof seems to be much more difficult.
\end{question}

\begin{exmp}[Hilbert torus]
Let $l^2(\R)=\{(x_n)\in\R\mid \sum_n |x_n|^2<\infty\}$ with the usual Hilbert product and consider the map $\pi:l^2(\R)\to \prod_{n=1}^\infty S^1=T^\infty$, $(x_n)\mapsto (\exp(2\pi i x_n))$. Consider the topology on $Y=\pi(l^2(\R))$ such that $\pi:l^2(\R)\to Y$ is a surjective local homeomorphism; in this way, a structure of Hilbert manifold is induced on $Y$. Even with this topology, $\pi_1(Y)=\bigoplus_{n=1}^\infty \Z$ and the cuplength is infinite, yielding $\Cat(Y)=\infty$.

A Hilbert lamination can be defined with the suspension of the homomorphism $\pi_1(Y)\to\operatorname{Diff}(S^1)$ given by $e_n\mapsto R_n$, where the elements $e_n$ form the standard base of $\bigoplus_{n=1}^\infty \Z$ and $R_n:S^1\to S^1$ is a rotation of the circle. Thus we have a structure of Hilbert lamination on the suspension $X=l^2(\R)\times_{(R_n)}S^1$. Of course, the usual Lebesgue measure on $S^1$ is an invariant measure for this lamination. Since the cuplength of $Y$ is infinite, it is clear that the leafwise cuplength of $X$ is also infinite. Therefore $\Cat(X)=\infty$ (independently of the choice of the rotations $R_n$).

However, if all the rotations are given by \Q-linear independent angles, then $\Cat(X,\LB)=0$, which can be seen as follows. Let $T^n=\prod_{i=1}^n S^1\times\prod_{i=n+1}^\infty\{(1,0)\}$. We can consider the usual lamination $\R^n\times_{(R_1,\dots,R_n)} S^1$ embedded in a natural way into $X$. Let $q:l^2(\R)\times S^1\to X$ be the quotient map. Thus an  open set in $X$ is a quotient of an open set in $l^2(\R)\times_{(R_n)}S^1$. Consider in the same way $q_n:\R^n\times S^1\to \R^n\times_{(R_1,\dots,R_n)} S^1$.

We knew that $\Cat(\R^n\times_{(R_1,\dots,R_n)} S^1,\LB)=0$. Then let $\{U_i^n\}_{i\in\N}$ be an open covering of $\R^n\times_{(R_1,\dots,R_n)} S^1$ such that $\sum_i \tau(U_i^n)<\frac{\varepsilon}{2^n}$. Consider the sets 
$$
\widetilde{U_i^n}=\{\pi((x_k)\times\{z\})\ |\ (x_1,...,x_n)\in q_n^{-1}(U_i^n)\,,\  | x_k |<1/2\ \forall k>n \}\;.
$$
They are open in $X$ and tangentially contractible. Therefore $\sum_i \tau(\widetilde{U_i^n})<\frac{\varepsilon}{2^n}$. The sets $\widetilde{U_i^n}$, $i,n\in\N$, form a covering of $\pi(T^\infty)$ since, for all $(z_n)\in Y$, there exists $N$ large enough such that $z_n$ is close to $(1,0)$ for all $n\geq N$. Thus the \LB-category of $X$ is zero.

Thus, in this Hilbert lamination, the number of critical sets of any $\omega$-PS map is always infinite, but it seems possible to find functions where the measure of the set of critical points is arbitrarily small.
\end{exmp}

\begin{ack}
This paper contains part of my PhD thesis, whose advisor is Prof.
Jes\'{u}s A. \'{A}lvarez L\'{o}pez.
\end{ack}


\begin{thebibliography}{99}

\bibitem{Alvarez-Candel} {\sc J.A.~\'Alvarez L\'opez, A.~Candel}. \textit{Equicontinuous foliated spaces}. Math.
Z.  263, 725--774 (2009).

\bibitem{Alvarez-Masa}{\sc J.A.~\'Alvarez L\'opez, X.M.~Masa}. \textit{Morphisms between complete Riemannian pseudogroups}. Topol. Appl., 155, 544--604 (2008).

\bibitem{Ballman} {\sc W. Ballman}. \textit{Closed geodesics on positively curved manifolds}. Ann. of Math. 116, 213--247 (1982).

\bibitem{Ballman-Thorbergsson-Ziller} {\sc  W. Ballman, G. Thorbergsson, W. Ziller}. \textit{Existence of closed geodesics}. J. Diff. Geom. 18, 221--252 (1983).

\bibitem{Candel-Conlon}{\sc A. Candel, L. Conlon}.
{\it Foliations I}. Amer. Math. Soc. (1999).

\bibitem{HellenColman}{\sc H. Colman Vale}.
{\it Categor\'{i}a L-S en foliaciones}. Tesis, Departamento de
Xeometr\'{i}a e Topolox\'{i}a, Universidad de Santiago de
Compostela (1998).

\bibitem{Macias}
{\sc H. Colman Vale, E. Mac\'ias Virg\'os}. {\it Tangential
Lusternik-Schnirelmann category of foliations}.  J. London Math.
Soc., 2, 745--756 (2002).

\bibitem{Connes}{\sc A. Connes}.
{\it A survey of foliations and operator algebras}. Proc.
Sympos. Pure Math., 38, 520--628 (1982).

\bibitem{Crandall-Pazy} {\sc M.G. Crandall, A. Pazy}. {\it Semi-groups of nonlinear contractions
and dissipative sets}. J. Functional Analysis, 3, 376--418, (1969).


\bibitem{Edwards-Millett-Sullivan}{\sc R. Edwards, K. Millett, D. Sullivan}. {\it Foliations with all leaves compact}. Topology
16, 13--32 (1977).

\bibitem{Epstein}{\sc D. B. A. Epstein}.
{\it Foliations with all leaves compact}. Ann. Inst. Fourier 26,
265--282 (1976).

\bibitem{Epstein2}{\sc D. B. A. Epstein}.
{\it Periodic flows on 3-manifolds.} Ann. Math. 95, 68--82
(1972).

\bibitem{Haefliger}{\sc A. Haefliger}. {\it Pseudogroups of local isometries}. Differential Geometry (Santiago de Compostela), L.A. Cordero, ed., Research Notes in Math. 131 (1984), Pitman Advanced Pub. Program, Boston, pp. 174--197, (1985).

\bibitem{James}{\sc I.M. James}.
{\it On category, in the sense of Lusternik-Schnirelmann}.
Topology, 17, 331--348 (1978).

\bibitem{Klingenberg} {\sc W. Klingenberg}. \textit{Lectures on closed geodesics}. A series of Comprehensive Studies in Mathematics, Springer-Verlag Berlin Heidelberg New York 230 (1978).

\bibitem{Kodama-Matsumoto}{\sc H. Kodama, S. Matsumoto}. {\it Minimal $C¹$-diffeomorphisms of the circle which admit measurable fundamental domains}. Preprint, arXiv:1005.0585v2.

\bibitem{Lusternik-Schnirelmann}
{\sc L. Lusternik, L. Scnirelmann}. {\it Methodes Topologiques
dans les Problemes Variationnets}. Herman, Paris (1934).

\bibitem{Menino1} {\sc C. Meni\~no}. {\it Transverse invariant measures extend to the ambient space}. Preprint, arXiv:1103.4696v1.

\bibitem{Menino2}{\sc C. Meni\~no}. {\it Tangential category and critical point theory}

\bibitem{Menino3} {\sc C. Meni\~no}. {\it Measurable versions of the LS category on laminations}. Preprint, arXiv:1108.5927v1.

\bibitem{Palais}
{\sc R.S. Palais}. {\it Lectures on Morse Theory}.  Lecture Notes, Brandies and Harvard Universities (1962).

\bibitem{Schwartz}
{\sc J.T. Schwartz}. {\it Generalizing the Lusternik-Schnirelmann Theory of Critical Points}. Comm. Pure
Appl. Math., 17, 307--315 (1964).

\bibitem{Vogt-Singhof}{\sc W. Singhof, E. Vogt.}
{\it Tangential category of foliations}. Topology 42, 603--627
(2003).

\bibitem{Takesaki}{\sc M. Takesaki}.
{\it Theory of Operator Algebras}. Springer-Verlag, New York,
Heidelberg, Berlin (1979).

\bibitem{Walczak}{\sc P.G. Walczak}.
{\it Dynamics of foliations, groups and pseudogroups.} Ed.
Birkhauser (2004).

\end{thebibliography}
\end{document}